\documentclass{amsart}

\usepackage{amsmath, amsfonts, mathrsfs}
\usepackage{amssymb,latexsym}
\usepackage{enumerate}
\usepackage{bbm}

\usepackage[top=2.4cm, bottom=2.4cm, left=2.7cm, right=2.7cm, includefoot, heightrounded]{geometry}

\newtheorem{theorem}{Theorem}[section]
\newtheorem{corollary}[theorem]{Corollary}
\newtheorem{lemma}[theorem]{Lemma}

\theoremstyle{definition}
\newtheorem{definition}[theorem]{Definition}

\numberwithin{equation}{section}

\begin{document}

\title[Dunford-Henstock-Kurzweil and Dunford-McShane Integrals]
{Dunford-Henstock-Kurzweil and Dunford-McShane Integrals \\ of Vector-Valued Functions defined on \\ $m$-dimensional bounded sets}

\author{ Sokol Bush Kaliaj }

\address{
Mathematics Department, 
Science Natural Faculty, 
University of Elbasan,
Elbasan, 
Albania.
}

\email{sokol\_bush@yahoo.co.uk}

\thanks{}



\begin{abstract} 
In this paper, we define the Dunford-Henstock-Kurzweil and 
the Dunford-McShane integrals of Banach space valued functions 
defined on a bounded Lebesgue measurable subset 
of $m$-dimensional Euclidean space $\mathbb{R}^{m}$. 
We will show that the new integrals are "natural" extensions of the McShane and the Henstock-Kurzweil integrals 
from $m$-dimensional closed non-degenerate intervals 
to $m$-dimensional bounded Lebesgue measurable sets.
As applications, 
we will present full descriptive characterizations of the McShane and Henstock-Kurzweil integrals in terms of our integrals. 
Moreover, a relationship between new integrals will be proved in terms of the Dunford integral.
\end{abstract}

\maketitle

\section{Introduction}

In the paper \cite{KAL1},  the Hake-Henstock-Kurzweil  and the Hake-McShane integrals 
are defined. 
It is proved that those integrals are "natural" extensions of 
the Henstock-Kurzweil and the McShane integrals 
from $m$-dimensional closed non-degenerate intervals 
to 
$m$-dimensional open and bounded sets, see Theorems 3.1 and 3.2 in \cite{KAL1}. 
The motivation behind those new integrals is to obtain Hake-type theorems 
for the Henstock-Kurzweil and the McShane integrals 
of a Banach space valued function defined on 
a closed non-degenerate interval in $m$-dimensional Euclidean space $\mathbb{R}^{m}$, 
see Theorems 3.3 and 3.4   in \cite{KAL1}.

In this paper, we define the Dunford-Henstock-Kurzweil and 
the Dunford-McShane integrals of Banach space valued functions 
defined on a bounded subsets $G \subset \mathbb{R}^{m}$ 
such that $|G \setminus G^{o}|=0$. 
We will show that the new integrals are also "natural" extensions of the McShane and the Henstock-Kurzweil integrals 
from $m$-dimensional closed non-degenerate intervals 
to $m$-dimensional bounded Lebesgue measurable sets, 
see Theorems \ref{T2.1} and \ref{T2.2}.

As applications, 
we will present full descriptive characterizations of the McShane 
and the Henstock-Kurzweil integrals 
in terms of our integrals, 
see Theorems \ref{T2.11} and \ref{T2.21}.

In the paper \cite{FRE1} D.~H.~Fremlin proved the following result 
for the case of a compact non-degenerate subinterval 
$I \subset \mathbb{R}$.

\begin{theorem}[Fremlin's Theorem] 
A function $f : I \to X$ is McShane integrable on $I$ 
if and only if it is Henstock-Kurzweil integrable and Pettis integrable on $I$.
\end{theorem}

Checking Fremlin's proof it can be seen that it still holds when $I$ is an 
$m$-dimensional closed non-degenerate subinterval in $\mathbb{R}^{m}$, 
c.f. Theorem 6.2.6 in \cite{SCH1}. 
By using Fremlin's Theorem, we will show  
a relationship between Dunford-McShane and Dunford-Henstock-Kurzweil integrals 
in terms of the Dunford integral, see Theorem \ref{T2.2}.

\section{Notations and Preliminaries}

Throughout this paper $X$ denotes a real Banach space with 
the norm $||\cdot||$ and $X^{*}$ its dual. 
The Euclidean space $\mathbb{R}^{m}$ is equipped with the maximum norm. 
$B_{m}(t,r)$ denotes the open ball in $\mathbb{R}^{m}$ with center $t$ 
and radius $r > 0$.
We denote by $\mathcal{L}(\mathbb{R}^{m})$ the $\sigma$-algebra of Lebesgue measurable subsets 
of $\mathbb{R}^{m}$ and by 
$\lambda$ the Lebesgue measure on $\mathcal{L}(\mathbb{R}^{m})$.  
$|A|$ denotes the Lebesgue measure of $A \in \mathcal{L}(\mathbb{R}^{m})$. 
We put
$$
\mathcal{L}(A) = 
\{ A \cap L : L \in \mathcal{L}(\mathbb{R}^{m}) \}, 
$$
for any $A \in \mathcal{L}(\mathbb{R}^{m})$.

The subset $\prod_{j=1}^{m} [a_{j},b_{j}] \subset \mathbb{R}^{m}$ 
is said to be a \textit{closed non-degenerate interval} in $\mathbb{R}^{m}$, 
if $-\infty < a_{j} < b_{j} < +\infty$,  for $j=1, \dotsc, m$.
Two closed non-degenerate intervals $I$ and $J$ in $\mathbb{R}^{m}$ are said to be 
\textit{non-overlapping} if 
$I^{o} \cap J^{o} = \emptyset$, where $I^{o}$ denotes the \textit{interior} of $I$. 
By $\mathcal{I}$ the family  of all closed non-degenerate subintervals in $\mathbb{R}^{m}$ is denoted 
and by $\mathcal{I}_{E}$ the family  of all closed non-degenerate 
subintervals in $E \in \mathcal{L}(\mathbb{R}^{m})$.

	%*******************************************%
	%---Additive functions HK partitions on E---%
	%*******************************************%

Let $E \in \mathcal{L}(\mathbb{R}^{m})$. 
A  function $F: \mathcal{I}_{E} \to X$  
is said to be an \textit{additive interval function}, 
if for each two non-overlapping intervals $I, J \in \mathcal{I}_{E}$ such that 
$I \cup J \in \mathcal{I}_{E}$, we have
\begin{equation*}
F(I \cup J) = F(I) + F(J).
\end{equation*}
A pair $(t, I)$ of a point $t \in E$  and an interval $I \in \mathcal{I}_{E}$ 
is called an \textit{$\mathcal{M}$-tagged interval} in $E$, $t$ is the tag of $I$. 
Requiring $t \in I$ for the tag of $I$ 
we get the concept of an \textit{$\mathcal{HK}$-tagged interval} in $E$. 
A finite collection $\{ (t_{i}, I_{i}) : i = 1, \dotsc, p \}$ of 
$\mathcal{M}$-tagged intervals ($\mathcal{HK}$-tagged intervals) in $E$ 
is called an  \textit{$\mathcal{M}$-partition} (\textit{$\mathcal{HK}$-partition}) in $E$, 
if $\{ I_{i} : i = 1, \dotsc, p \}$ is a collection of pairwise non-overlapping intervals 
in $\mathcal{I}_{E}$. 
Given $Z \subset E$, a positive function $\delta: Z \to (0,+\infty)$ 
is called a \textit{gauge} on $Z$. 
We say that an 
$\mathcal{M}$-partition ($\mathcal{HK}$-partition) $\pi = \{ (t_{i}, I_{i}) : i = 1, \dotsc, p \}$ 
in $E$ is 
\begin{itemize}
\item
$\mathcal{M}$-partition ($\mathcal{HK}$-partition) of $E$, if $\cup_{i=1}^{p} I_{i} = E$,
\item
$Z$-tagged if $\{t_{1}, \dotsc, t_{p} \} \subset Z$,
\item 
$\delta$-fine if for each $i = 1, \dotsc, p$, we have 
$
I_{i} \subset B_{m}(t_{i},\delta(t_{i})). 
$
\end{itemize}

We now recall the definitions of the McShane and the Henstock-Kurzweil integrals of a function 
$f:W \to X$, where $W$ is a fixed interval in $\mathcal{I}$. 
The function $f$ is said to be 
\textit{McShane (Henstock-Kurzweil) integrable} on $W$ 
if there is a vector 
$x_{f} \in X$ such that for every $\varepsilon>0$, 
there exists a gauge $\delta$ on $W$  
such that for every $\delta$-fine $\mathcal{M}$-partition ($\mathcal{HK}$-partition) 
$\pi$ of $W$, we have 
$$
||\sum_{(t,I) \in \pi} f(t)|I| - x_{f}|| < \varepsilon.
$$
In this case, the vector $x_{f}$ is said to be the  
\textit{McShane (Henstock-Kurzweil) integral} of $f$ on $W$ 
and we set $x_{f}=(M)\int_{W} f d \lambda$ ($x_{f}=(HK)\int_{W} f d \lambda$). 
The function $f$ is said to be 
\textit{McShane (Henstock-Kurzweil) integrable} over a subset $A \subset W$, 
if the function $f . \mathbbm{1}_{A} : W \to X$ is McShane (Henstock-Kurzweil) integrable on $W$, 
where $\mathbbm{1}_{A}$ is the characteristic function of the set $A$. 
The McShane (Henstock-Kurzweil) integral of $f$ over $A$ will be denoted by $(M)\int_{A}f d\lambda$ 
($(HK)\int_{A}f d\lambda$).
If $f:W \to X$ is McShane integrable on $W$, then by Theorem 4.1.6 in \cite{SCH1} 
the function $f$ is the McShane integrable on each $A \in \mathcal{L}(W)$, 
while by Theorem 3.3.4 in \cite{SCH1}, 
if $f$ is Henstock-Kurzweil integrable on $W$, then $f$  
is the Henstock-Kurzweil integrable on each $I \in \mathcal{I}_{W}$. 
Therefore, we can define an additive interval function 
$F: \mathcal{I}_{W} \to X$ as follows 
$$
F(I) = (M) \int_{I} f d \lambda,~  (~ F(I) = (HK) \int_{I} f d \lambda ~),~
\text{ for all }I \in \mathcal{I}_{W},
$$
which is called the primitive of $f$.

The basic properties of the McShane integral and the Henstock-Kurzweil integral 
can be found in \cite{BON}, \cite{CAO}, \cite{DIP}, \cite{FRE1}-\cite{FRE3}, \cite{GORD1}-\cite{GORD3},  
\cite{GUO1}, \cite{GUO2}, 
\cite{KURZ}, \cite{LEE1}, \cite{LEE2} and \cite{SCH1}. 
We do not present them here. 
The reader is referred to the above mentioned references for the details.

	%*********************************************%
	%---Hake-functions and Negligible Variation---%
	%*********************************************%

We now define the Dunford-McShane  and the Dunford-Henstock-Kurzweil integrals 
on a bounded Lebesgue measurable subset in $\mathbb{R}^{m}$. 
Fix a bounded Lebesgue measurable subset  $E \in \mathcal{L}(\mathbb{R}^{m})$ such that 
$E^{o} \neq \emptyset$, where $E^{o}$ is the interior of $E$. 
A sequence $(I_{k})$ of pairwise non-overlapping intervals in $\mathcal{I}_{E}$ is said to be a 
\textit{division in} $E$. 
By $\mathscr{P}_{E}$ the family of all divisions in $E$ is denoted. 
A division $(I_{k}) \in \mathscr{P}_{E}$ is said to be a 
\textit{division of} $E$ if  
$$
E = \bigcup_{k=1}^{+\infty} I_{k}.
$$   
We denote by $\mathscr{D}_{E}$ the family of all divisions of $E$.   
Clearly, $\mathscr{D}_{E} \subset \mathscr{P}_{E}$.
By Lemma 2.43 in \cite{FOLL}, the family $\mathscr{D}_{E^{o}}$ is not empty,
and since 
$$
\mathscr{D}_{E^{o}} \subset \mathscr{P}_{E^{o}} \subset \mathscr{P}_{E},
$$
it follows that $\mathscr{P}_{E}$ is not empty. 
\begin{definition}
A additive interval function  $F : \mathcal{I}_{E} \to X$  
is said to be a \textit{Dunford-function}, 
if given a division $(I_{k}) \in \mathscr{P}_{E}$, we have
\begin{itemize}
\item
the series 
$$
\sum_{k: |I \cap I_{k}| > 0 } F(I \cap I_{k})
$$ 
is unconditionally convergent in $X$, 
for each $I \in \mathcal{I}$,
\item
if $(I_{k}) \in \mathscr{D}_{E^{o}}$, then the equality 
$$
F(I) = \sum_{k: |I \cap I_{k}| > 0 } F(I \cap I_{k}),
$$ 
holds for all $I \in \mathcal{I}_{E}$. 
\end{itemize} 
\end{definition}

\begin{definition}
We say that the additive interval function $F : \mathcal{I}_{E} \to X$ 
has \textit{$\mathcal{M}$-negligible variation 
($\mathcal{HK}$-negligible variation) over} a subset $Z \subset \mathbb{R}^{m}$, 
if for each $\varepsilon >0$ there exists a gauge 
$\delta_{\varepsilon}$ on $Z$ such that 
for each $Z$-tagged $\delta_{\varepsilon}$-fine $\mathcal{M}$-partition ($\mathcal{HK}$-partition) $\pi$ in $\mathbb{R}^{m}$, we have
\begin{itemize}
\item
the series 
$$
\sum_{k: |I \cap I_{k}| > 0 } F(I \cap I_{k})
$$ 
is unconditionally convergent in $X$, for each $(t,I) \in \pi$, 
\item
the inequality 
$$
||\sum_{(t,I) \in \pi} 
\left (
\sum_{k: |I \cap I_{k}| > 0 } F(I \cap I_{k})
\right ) 
|| < \varepsilon,
$$
holds, 
\end{itemize}
whenever $(I_{k}) \in \mathscr{D}_{E^{o}}$. 
We say that $F$ has  
\textit{$\mathcal{M}$-negligible variation ($\mathcal{HK}$-negligible variation) outside} of $E^{o}$ 
if $F$ has $\mathcal{M}$-negligible variation ($\mathcal{HK}$-negligible variation) 
over $(E^{o})^{c} = \mathbb{R}^{m} \setminus E^{o}$.
\end{definition}

	%*************************************************%
	%---HMcShane (HHenstock-Kurzweil) Integral on G---%
	%*************************************************%

\begin{definition}\label{DH}
We say that a function $f: E \to X$ is \textit{Dunford-McShane (Dunford-Henstock-Kurzweil) integrable} on $E$ with 
the primitive $F: \mathcal{I}_{E} \to X$,  
if we have
\begin{itemize}
\item
for each $\varepsilon >0$ there exists a gauge $\delta_{\varepsilon}$ on $E^{o}$ such that  
for each $\delta_{\varepsilon}$-fine $\mathcal{M}$-partition ($\mathcal{HK}$-partition) $\pi$ in 
$E^{o}$, we have
$$
|| \sum_{(t,I) \in \pi } 
( f(t)|I| - F(I) ) || < \varepsilon,   
$$
\item
$F$ is a Dunford-function, 
\item
$F$ has $\mathcal{M}$-negligible ($\mathcal{HK}$-negligible) variation outside of $E^{o}$.
\end{itemize}
In this case, we define the Dunford-McShane (the Dunford-Henstock-Kurzweil) integral of $f$ over $I$ as follows
\begin{equation*}
(DM) \int_{I} f d \lambda = F(I), \quad \left ( (DHK) \int_{I} f d \lambda = F(I) \right ).
\end{equation*} 
\end{definition}
Clearly, if $f: E \to X$ is Dunford-McShane (Dunford-Henstock-Kurzweil) integrable on $E$ with 
the primitive $F$ and $E = E^{o}$, 
then $f$ is Hake-McShane (Hake-Henstock-Kurzweil) integrable on $E$ with 
the primitive $F$.

Finally, we recall the definition of the Dunford integral 
in the second dual $X^{**}$ of $X$, c.f. \cite{DIES}. 
A function $f:E \to X$ is said to be \textit{Dunford integrable}, 
if $x^{*}f$ is Lebesgue integrable (or, equivalently McShane integrable) for all $x^{*} \in X^{*}$. 
In the case that $f$ is Dunford integrable, by Dunford's Lemma, 
for each $A \in \mathcal{L}(E)$, there exists 
$x^{**}_{A} \in X^{**}$ satisfying  
$$
x^{**}_{A}(x^{*}) = (M)\int_{A} x^{*}f d \lambda, 
\text{ for all }x^{*} \in X^{*},
$$
and we write $x^{**}_{A} = (D)\int_{A} f d\lambda$. 

If $(D)\int_{A} f d\lambda \in e(X) \subset X^{**}$, for all $A \in \mathcal{L}(E)$, 
then $f$ is called \textit{Pettis integrable}, where $e$ is the \textit{canonical embedding} of $X$ into $X^{**}$. 
In this case, we write $(P)\int_{A} f d\lambda$ instead of $(D)\int_{A} f d\lambda$ 
to denote \textit{Pettis integral} of $f$ over $A \in \mathcal{L}(E)$.

\section{The Main results}

From now on $G$ will be a bounded subset of  $\mathbb{R}^{m}$ such that 
$G^{o} \neq \emptyset$ and $|G \setminus G^{o}|=0$.  
Since $G$ is a bounded subset of $\mathbb{R}^{m}$, 
we can fix an interval $I_{0} \in \mathcal{I}$ such that $G \subset I_{0}$. 
Given a function $f:G \to X$, we denote by $f_{0}: I_{0} \to X$ the function defined as follows
\begin{equation*}
f_{0}(t) = 
\left \{
\begin{array}{ll}
f(t) & \text{ if }t \in G \\
0 & \text{ if }t \in I_{0} \setminus G. 
\end{array}
\right.
\end{equation*} 
Assume that the functions $f : G \to X$ and $F: \mathcal{I}_{G} \to X$ are given. 
Then, given a division $(C_{k}) \in \mathscr{P}_{G}$, 
we denote
$$
f_{k} = f \vert_{C_{k}}\text{ and }
F_{k} = F \vert_{\mathcal{I}_{C_{k}}},
\text{ for each }k \in \mathbb{N}.
$$  

Let us start with the following auxiliary lemmas.

\begin{lemma}\label{L2.1}
Let $f: G \to X$ be a function. 
Then, given $\varepsilon >0$ there exists a gauge $\delta$ on $Z = G \setminus G^{o}$ 
such that for each $\delta$-fine $Z$-tagged $\mathcal{M}$-partition 
(or $\mathcal{HK}$-partition) $\pi$ in $I_{0}$ we have
$$
|| \sum_{(t,I) \in \pi} f(t)|I| ~|| < \varepsilon.
$$ 
\end{lemma}
\begin{proof}
Define a function $g_{0} : I_{0} \to X$ as follows
\begin{equation*}
g_{0}(t) = 
\left \{
\begin{array}{ll}
f(t) & \text{ if }t \in Z \\
0 & \text{ otherwise }. 
\end{array}
\right.
\end{equation*} 
Then, by Theorem 3.3.1 ~(or Corollary 3.3.2) in \cite{SCH1}, 
$g_{0}$ is McShane (or Henstock-Kurzweil) integrable on $I_{0}$ and
$$
(M)\int_{I} g_{0} d \lambda = 0 \quad ((HK)\int_{I} g_{0} d \lambda = 0~), \text{ for all }I \in \mathcal{I}_{I_{0}}. 
$$
Therefore, by Lemma 3.4.2 ~(or Lemma 3.4.1) in \cite{SCH1}, 
given $\varepsilon >0$ there exists a gauge $\delta$ on $Z$  
such that for each $\delta$-fine $Z$-tagged $\mathcal{M}$-partition~ 
(or $\mathcal{HK}$-partition) $\pi$ in $I_{0}$ we have
$$
|| \sum_{(t,I) \in \pi} g_{0}(t)|I| ~|| < \varepsilon,
$$ 
and since 
$g_{0}(t) = f(t)$ for all $t \in Z$, the last result proves the lemma.
\end{proof}

\begin{lemma}\label{L2.2}
Let $f: G \to X$ be a function, 
and let $F: \mathcal{I}_{G} \to X$ be an additive interval function. 
If $F$ has $\mathcal{M}$-negligible~ (or $\mathcal{HK}$-negligible) variation outside of $G^{o}$, then
given $\varepsilon >0$ there exists a gauge $\delta$ on $Z = G \setminus G^{o}$ 
such that for each $\delta$-fine $Z$-tagged $\mathcal{M}$-partition~(or $\mathcal{HK}$-partition) 
$\pi$ in $I_{0}$ we have
$$
|| \sum_{(t,I) \in \pi} 
\left ( ~f(t)|I| - \sum_{k:|I \cap I_{k}| >0} F(I \cap I_{k})~ \right ) 
|| < \varepsilon,
$$ 
whenever $(I_{k}) \in \mathscr{D}_{G^{o}}$. 
\end{lemma}
\begin{proof}
Since $F$ has $\mathcal{M}$-negligible ($\mathcal{HK}$-negligible) variation 
outside of $G^{o}$, 
given $\varepsilon >0$ there exists a gauge $\delta_{v}$ 
on $(G^{o})^{c}$  
such that for each $\delta$-fine $(G^{o})^{c}$-tagged $\mathcal{M}$-partition ($\mathcal{HK}$-partition) 
$\pi$ in $I_{0}$, 
we have
$$
|| \sum_{(t,I) \in \pi} 
\left ( \sum_{k:|I \cap I_{k}| >0} F(I \cap I_{k})~ \right ) 
|| < \frac{\varepsilon}{2},
$$ 
whenever $(I_{k}) \in \mathscr{D}_{G^{o}}$. 

By Lemma \ref{L2.1}, 
there exists a gauge $\delta_{0}$ on $Z$ 
such that for each $\delta_{0}$-fine $Z$-tagged $\mathcal{M}$-partition ($\mathcal{HK}$-partition) 
$\pi$ in $I_{0}$, we have
$$
|| \sum_{(t,I) \in \pi} f(t)|I| ~|| < \frac{\varepsilon}{2}.
$$ 
Define a gauge $\delta$ on $Z$ by $\delta(t) = \min\{\delta_{v}(t), \delta_{0}(t)\}$ for all $t \in Z$. 
Let $\pi$ be a $\delta$-fine $Z$-tagged $\mathcal{M}$-partition ($\mathcal{HK}$-partition) $\pi$ in $I_{0}$. 
Then,
\begin{equation*}
\begin{split}
|| \sum_{(t,I) \in \pi} 
\left ( ~f(t)|I| - \sum_{k:|I \cap I_{k}| >0} F(I \cap I_{k})~ \right ) 
||
& \leq	
|| \sum_{(t,I) \in \pi} f(t)|I| ~|| \\
&+
|| \sum_{(t,I) \in \pi} 
\left (\sum_{k:|I \cap I_{k}| >0} F(I \cap I_{k})~ \right ) 
|| < \frac{\varepsilon}{2} + \frac{\varepsilon}{2} = \varepsilon,
\end{split}
\end{equation*}
and this ends the proof.
\end{proof}

\begin{theorem}\label{T2.1}
Let $f: G \to X$ be a function and let $F: \mathcal{I}_{G} \to X$ be an additive interval function.  
Then, the following statements are equivalent:
\begin{itemize}
\item[(i)]
$f$ is Dunford-McShane integrable on $G$ with the primitive $F$,
\item[(ii)] 
$f_{0}$ is McShane integrable on $I_{0}$ with the primitive $F_{0}$ 
such that $F_{0}(I) = F(I)$, for all $I \in \mathcal{I}_{G}$,
\item[(iii)]
$F$ is a Dunford-function and has $\mathcal{M}$-negligible variation outside of $G^{o}$, 
and given any division $(C_{k}) \in \mathscr{D}_{G^{o}}$, 
each $f_{k}$ is McShane  integrable on $C_{k}$ with the primitive $F_{k}$. 
\end{itemize}
\end{theorem}
\begin{proof}
$(i) \Rightarrow (iii)$ 
Assume that $f$ is Dunford-McShane integrable on $G$ with the primitive $F$ 
and let $(C_{k})$ be any division of $G^{o}$. 
Then, given $\varepsilon >0$ there exists a gauge $\delta$ on $G^{o}$ such that 
for each $\delta$-fine $\mathcal{M}$-partition $\pi$ in $G^{o}$, we have
\begin{equation*}
|| \sum_{(t,I) \in \pi } (~ f(t)|I| - F(I)~) || < \varepsilon.   
\end{equation*}

By Definition \ref{DH},
$F$ is a Dunford-function and has $\mathcal{M}$-negligible variation outside of $G^{o}$.  
Thus, it remains to prove that each $f_{k}$ 
is McShane integrable on $C_{k}$ 
with the primitive $F_{k}$. 
Let $\pi_{k}$ be a $\delta_{k}$-fine $\mathcal{M}$-partition of $C_{k}$, 
where $\delta_{k} = \delta \vert_{C_{k}}$. 
Then, $\pi_{k}$ is a $\delta$-fine $\mathcal{M}$-partition in $G^{o}$
and, therefore, 
\begin{equation*}
\begin{split}
||\sum_{(t,I) \in \pi_{k}} (~  f_{k}(t). |I|-F_{k}(I) ~)|| 
&= || \sum_{(t,I) \in \pi_{k}} (~ f(t). |I| - F(I) ~) || < \varepsilon. 
\end{split} 
\end{equation*}
This means that $f_{k}$ is McShane integrable on $C_{k}$ with the primitive $F_{k}$.

$(iii) \Rightarrow (ii)$ 
Assume that $(iii)$ holds. Let $\varepsilon >0$ and 
let $(C_{k}) \in \mathscr{D}_{G^{o}}$. 
Then, since  
each function $f_{k}$ is McShane integrable on $C_{k}$ 
with the primitive $F_{k}$,  
by Lemma 3.4.2 in \cite{SCH1}, 
there exists a gauge $\delta_{k}$ on $C_{k}$ such that 
for each $\delta_{k}$-fine $\mathcal{M}$-partition $\pi_{k}$ in $C_{k}$, we have
\begin{equation}\label{eq_SaksHen}
||\sum_{(t,I) \in \pi_{k}} (~ f_{k}(t)|I| - F_{k}(I)~) || \leq 
\frac{1}{2^{k}}\frac{\varepsilon}{4}.  
\end{equation}

Note that for any $t \in  G^{o} = \cup_{k} C_{k}$, we have the following possible cases:
\begin{itemize}
\item
there exists $i_{0} \in \mathbb{N}$ such that $t \in (C_{i_{0}})^{o}$, 
\item
there exists $j_{0} \in \mathbb{N}$ such that $t \in C_{j_{0}} \setminus (C_{j_{0}})^{o}$. 
In this case, 
there exists a finite set 
$\mathcal{N}_{t} = \{j \in \mathbb{N}: t \in C_{j} \setminus (C_{j})^{o} \}$ such that 
$t \in \bigcap_{j \in \mathcal{N}_{t}} C_{j}$ and $t \notin C_{k}$,  
for all $k \in \mathbb{N} \setminus \mathcal{N}_{t}$. 
Hence, $t \in ( \cup_{j \in \mathcal{N}_{t}} C_{j} )^{o}$. 
\end{itemize} 
For each $k \in \mathbb{N}$, choose $\delta_{k}$ so that for any $t \in G^{o}$, we have
$$
t \in (C_{k})^{0} \Rightarrow B_{m}(t, \delta_{k}(t)) \subset C_{k}
$$
and
$$
t \in C_{k} \setminus (C_{k})^{o} \Rightarrow  
B_{m}(t, \delta_{k}(t)) \subset \bigcup_{j \in \mathcal{N}_{t}} C_{j}.
$$

Since $F$ has $\mathcal{M}$-negligible variation outside of $G^{o}$, 
there exists a gauge $\delta_{v}$ on $I_{0} \setminus G^{o}$ 
such that for each $(I_{0} \setminus G^{o})$-tagged 
$\delta_{v}$-fine $\mathcal{M}$-partition $\pi_{v}$ in $I_{0}$, we have
\begin{equation}\label{eq_Neg_Variation2}
||\sum_{(t,I) \in \pi_{v}} 
\left (  \sum_{k:|I \cap C_{k}|>0}  F(I \cap C_{k}) \right ) || < 
\frac{\varepsilon}{4}. 
\end{equation}

By Lemma \ref{L2.2}, we can choose $\delta_{v}$ so that 
for each $(G \setminus G^{o})$-tagged 
$\delta_{v}$-fine $\mathcal{M}$-partition $\pi$ in $I_{0}$, we have
\begin{equation}\label{eq_Neg_Variation3}
|| \sum_{(t,I) \in \pi} 
\left ( ~f(t)|I| - \sum_{k:|I \cap C_{k}| >0} F(I \cap C_{k})~ \right ) 
|| < \frac{\varepsilon}{4}.
\end{equation}

By  hypothesis, we have also that $F$ is a Dunford-function.  
Therefore, we can define an additive interval function 
$F_{0}: \mathcal{I}_{I_{0}} \to X$ as follows
\begin{equation}\label{eqF0} 
F_{0}(I) = \sum_{k: |I \cap C_{k}|>0} F(I \cap C_{k}), 
\text{ for all } I \in \mathcal{I}_{I_{0}}. 
\end{equation} 
Clearly, 
$F_{0}(I) = F(I)$, for all $I \in \mathcal{I}_{G}$. 
We will show that $f_{0}$ is McShane integrable on $I_{0}$ 
with the primitive $F_{0}$. 
To see this, we first define a gauge $\delta_{0} : I_{0} \to (0,+\infty)$ as follows.  
For each $t \in G^{o}$, we choose 
$$
\delta_{0}(t) 
= 
\left \{
\begin{array}{ll}
\delta_{i_{0}}(t) & \text{ if }t \in (C_{i_{0}})^{o}  \\
\min \{ \delta_{j}(t): j \in \mathcal{N}_{t} \} & \text{ otherwise},
\end{array}
\right.
$$
while for $t \in I_{0} \setminus G^{o}$, $\delta_{0}(t)= \delta_{v}(t)$.
Let $\pi$ be an arbitrary $\delta_{0}$-fine $\mathcal{M}$-partition of $I_{0}$. 
Then, 
$$
\pi = \pi_{1} \cup \pi_{2} \cup \pi_{3} \cup \pi_{4},
$$ 
where 
\begin{equation*}
\begin{split}
\pi_{1} 
&= \{ (t,I) \in \pi : (\exists i_{0} \in \mathbb{N})[t \in (C_{i_{0}})^{o}] \} \\
\pi_{2} 
&= \{(t,I) \in \pi : 
(\exists j_{0} \in \mathbb{N})[t \in C_{j_{0}} \setminus (C_{j_{0}})^{o}] \}
\end{split}
\end{equation*}
and
\begin{equation*}
\begin{split}
\pi_{3} 
&= \{(t,I) \in \pi : t \in G \setminus G^{o} \} \\
\pi_{4} 
&= \{(t,I) \in \pi : t \in I_{0} \setminus G \}.
\end{split}
\end{equation*}
Hence,    
\begin{equation}\label{eq_SS.1}
\begin{split}
||\sum_{(t,I) \in \pi} (~f_{0}(t)|I| - F_{0}(I)~) &|| 
\leq
||\sum_{(t,I) \in \pi_{1}} (~f(t)|I| - F(I)~) || \\ 
+ 
||\sum_{(t,I) \in \pi_{2}} (~f(t)|I| - F(I)~) || 
+&
||\sum_{(t,I) \in \pi_{3}} (~f(t)|I| - F_{0}(I)~)||
+
||\sum_{(t,I) \in \pi_{4}}  F_{0}(I)||.
\end{split}
\end{equation} 
Note that, if we define
\begin{equation*}
\begin{split}
\pi_{1}^{k} &= \{(t, I) : (t,I) \in \pi_{1}, t \in (C_{k})^{o} \}, \\ 
\pi_{2}^{k} &= \{(t, I \cap C_{k}) : (t,I) \in \pi_{2}, 
t \in C_{k} \setminus (C_{k})^{o}, |I \cap C_{k}|>0 \},
\end{split}
\end{equation*}
then $\pi_{1}^{k}$ and $\pi_{2}^{k}$ are $\delta_{k}$-fine $\mathcal{M}$-partitions in $C_{k}$. 
Therefore, by \eqref{eq_SaksHen}, it follows that
\begin{equation}\label{eq_SS.2}
\begin{split}
||\sum_{(t,I) \in \pi_{1}} (~f(t)|I| - F(I)~) || 
&=
||\sum_{k} \sum_{\underset{t \in (C_{k})^{0}}{(t,I) \in \pi_{1}}} 
(~f(t)|I| - F(I)~)|| \\
&\leq
\sum_{k} ||\sum_{(t,I) \in \pi_{1}^{k}} 
(~f_{k}(t)|I| - F_{k}(I)~)|| \\
&\leq
\sum_{k=1}^{+\infty} \frac{1}{2^{k}}\frac{\varepsilon}{4} = \frac{\varepsilon}{4}
\end{split}
\end{equation}
and  
\begin{equation}\label{eq_SS.3}
\begin{split}
||\sum_{(t,I) \in \pi_{2}} (~f(t)|I| - F(I)~) || 
&= 
||\sum_{(t,I) \in \pi_{2}} 
\left ( \sum_{\underset{| I \cap C_{j}| >0}{j \in \mathcal{N}_{t}}} 
(~ f(t) . | I \cap C_{j}| - F(I \cap C_{j}) ~) \right ) || \\
=& 
||\sum_{(t,I) \in \pi_{2}} 
\left ( \sum_{\underset{| I \cap C_{j}| >0}{j \in \mathcal{N}_{t}}}
( f_{j}(t) . | I \cap C_{j}| - F_{j}(I \cap C_{j}) ) \right ) || \\
=&
||\sum_{k} 
\left ( 
\sum_{(t,I) \in \pi_{2}^{k}} 
(~f_{k}(t) . | I \cap C_{k}| - F_{k}(I \cap C_{k})~)  
\right ) ||\\
\leq&
\sum_{k} 
||
\left ( 
\sum_{(t,I) \in \pi_{2}^{k}} 
(~f_{k}(t) . | I \cap C_{k}| - F_{k}(I \cap C_{k})~)  
\right ) || \\
\leq&
\sum_{k=1}^{+\infty} \frac{1}{2^{k}}\frac{\varepsilon}{4} = \frac{\varepsilon}{4}.
\end{split}
\end{equation}

We have also that $\pi_{3}$ is a $(G \setminus G^{o})$-tagged 
$\delta_{v}$-fine $\mathcal{M}$-partition in $I_{0}$ 
and 
$\pi_{4}$ is $(I_{0} \setminus G)$-tagged 
$\delta_{v}$-fine $\mathcal{M}$-partition in $I_{0}$. 
Therefore, \eqref{eq_Neg_Variation2} and \eqref{eq_Neg_Variation3} together with 
\eqref{eqF0} yield  
$$
||\sum_{(t,I) \in \pi_{3}} (~f(t)|I| - F_{0}(I)~)|| < \frac{\varepsilon}{4} 
\quad \text{ and } \quad
|| \sum_{(t,I) \in \pi_{4}}  F_{0}(I) ||   < \frac{\varepsilon}{4}.
$$ 

The last result together with \eqref{eq_SS.1}, \eqref{eq_SS.2} and 
\eqref{eq_SS.3} yields
$$
||\sum_{(t,I) \in \pi} (~f_{0}(t) . |I| - F_{0}(I)~)|| < \varepsilon,
$$ 
and since 
$\pi$ was an arbitrary $\delta_{0}$-fine $\mathcal{M}$-partition of $I_{0}$, 
it follows that $f_{0}$ is McShane integrable on $I_{0}$ 
with the primitive $F_{0}$.

$(ii) \Rightarrow (i)$ 
Assume that $(ii)$ holds.   
Then, by Lemma 3.4.2 in \cite{SCH1}, 
given $\varepsilon >0$ there exists a gauge $\delta_{0}$ on $I_{0}$ such that 
for each $\delta_{0}$-fine $\mathcal{M}$-partition $\pi_{0}$ in $I_{0}$, we have
\begin{equation}\label{eq_SaksMM}
||\sum_{(t,I) \in \pi } (~ f_{0}(t)|I| - F_{0}(I) ~)|| < \varepsilon.   
\end{equation}
The gauge $\delta_{0}$ can be chosen 
so that for each $t \in I_{0}$, we have
$$
t \in G^{o} \Rightarrow B_{m}(t,\delta_{0}(t)) \subset G^{o}.
$$
By Lemma \ref{L2.1}, we can also choose $\delta_{0}$ 
so that 
\begin{equation}\label{eq_SaksMMM}
||\sum_{(t,I) \in \pi } f_{0}(t)|I|~|| < \varepsilon,   
\end{equation}
whenever $\pi$ is $(G \setminus G^{o})$-tagged $\delta_{0}$-fine 
$\mathcal{M}$-partition in $I_{0}$.

Hence, if we define $\delta = \delta_{0} \vert_{G^{o}}$, 
then for each 
$\delta$-fine $\mathcal{M}$-partition $\pi$ in $G^{o}$, we have
\begin{equation*}
||\sum_{(t,I) \in \pi } (~ f(t)|I| - F(I) ~)|| < \varepsilon.   
\end{equation*}
Thus, it remains to show that $F$ is a Dunford-function 
and has $\mathcal{M}$-negligible variation outside of $G^{o}$.

We first show that $F$ is a Dunford-function. 
Let $(I_{k}) \in \mathscr{P}_{G}$.

Since for any $I \in \mathcal{I}$, we have
\begin{equation*}
\begin{split}
\sum_{k:| I \cap I_{k}| >0}  F(I \cap I_{k}) 
=& \sum_{k:| I \cap I_{k}| >0}  F_{0}(I \cap I_{k}) 
= \sum_{k:| I \cap I_{k}| >0}  (M) \int_{I \cap I_{k}} f_{0} d \lambda \\
=& (M) \int_{I \cap ( \cup_{k} I_{k} )} f_{0} d \lambda,  
\end{split}
\end{equation*}
the series
$\sum_{k:| I \cap I_{k}| >0}  F(I \cap I_{k})$ is unconditionally  convergent in $X$.

If $(I_{k}) \in \mathscr{D}_{G^{o}}$ and $I \in \mathcal{I}_{G}$,  
then  
\begin{equation*}
\begin{split}
\sum_{k:| I \cap I_{k}| >0}  F(I \cap I_{k}) 
&=(M) \int_{I \cap G^{o}} f_{0} d \lambda =(M) \int_{I \cap G} f_{0} d \lambda
\\
&= (M) \int_{I} f_{0} d \lambda 
= F_{0}(I) = F(I). 
\end{split}
\end{equation*} 
Thus, $F$ is a Dunford-function.

Finally, we show that $F$ has $\mathcal{M}$-negligible variation outside of $G^{o}$.
To see this, 
we define a gauge $\delta_{v}$ on $(G^{o})^{c}$ 
by $\delta_{v}(t) = \delta_{0}(t)$ if $t \in I_{0} \setminus G^{o}$, 
while for $t \notin I_{0}$, we choose $\delta_{v}(t)$ so that 
$B_{m}(t, \delta_{v}(t) ) \cap I_{0} = \emptyset$. 
Assume that $\pi_{v}$ is a  $(G^{o})^{c}$-tagged $\delta_{v}$-fine 
$\mathcal{M}$-partition in $\mathbb{R}^{m}$. 
Hence, 
$$
\pi_{0} = 
\{ (t,I \cap I_{0}): (t,I) \in \pi_{v}, t \in I_{0} \setminus G^{o}, |I \cap I_{0}| >0 \}
$$
is a $\delta_{0}$-fine $\mathcal{M}$-partition in $I_{0}$. 
Note that $\pi_{0} = \pi_{a} \cup \pi_{b}$, where
$$
\pi_{a} = 
\{ (t,J) \in  \pi_{0}: t \in I_{0} \setminus G \}
\text{ and }
\pi_{b} = 
\{ (t,J) \in  \pi_{0}: t \in (G \setminus G^{o})\}.
$$
Since $\pi_{a}$ and $\pi_{b}$ are $\delta_{0}$-fine $\mathcal{M}$-partitions in $I_{0}$, 
by \eqref{eq_SaksMM} and \eqref{eq_SaksMMM},  
it follows that
\begin{equation*}
\begin{split} 
&|| \sum_{(t,J) \in \pi_{0}}   F_{0}(J) || 
\leq
|| \sum_{(t,J) \in \pi_{a}}   F_{0}(J) ||  + 
|| \sum_{(t,J) \in \pi_{b}}   F_{0}(J) || \\
&\leq
|| \sum_{(t,J) \in \pi_{a}}   (f_{0}(t)|J| - F_{0}(J))  ||  + 
|| \sum_{(t,J) \in \pi_{b}}   (f_{0}(t)|J| - F_{0}(J))  || \\
&+|| \sum_{(t,J) \in \pi_{b}}   f_{0}(t)|J|~ || < 3\varepsilon.
\end{split}
\end{equation*} 
On the other hand, we have also
\begin{equation*}
\begin{split} 
|| \sum_{(t,J) \in \pi_{0}}   F_{0}(J) || 
&= || \sum_{(t,J) \in \pi_{0}} (M)\int_{J} f_{0}d \lambda || \\
&=|| \sum_{(t,J) \in \pi_{0}} 
\left ( 
(M)\int_{J \setminus G} f_{0}d \lambda  + 
(M)\int_{J \cap (G \setminus G^{o})} f_{0}d \lambda +
(M)\int_{J \cap G^{o}} f_{0}d \lambda \right ) || \\
&=|| \sum_{(t,J) \in \pi_{0}} 
(M)\int_{J \cap G^{o}} f_{0}d \lambda  || 
=  
|| \sum_{(t,J) \in \pi_{0}} 
\left ( \sum_{k=1}^{+\infty} (M)\int_{J \cap I_{k}} f_{0}d \lambda  \right )||\\
&=|| \sum_{(t,J) \in \pi_{0}} 
\left ( \sum_{k:|J \cap I_{k}|>0} F_{0}(J \cap I_{k}) \right )|| 
=|| \sum_{(t,J) \in \pi_{0}} 
\left ( \sum_{k:|J \cap I_{k}|>0} F(J \cap I_{k}) \right )||\\
&=|| \sum_{(t,I) \in \pi_{v}} \left ( \sum_{k:|I \cap I_{k}|>0} F(I \cap I_{k}) \right )||,
\end{split}
\end{equation*} 
whenever $(I_{k}) \in \mathscr{D}_{G^{o}}$. 
It follows that
\begin{equation*}
\begin{split} 
|| \sum_{(t,I) \in \pi_{v}} \left ( \sum_{k:|I \cap I_{k}|>0} F(I \cap I_{k}) \right )|| 
< 3\varepsilon. 
\end{split}
\end{equation*} 
This means that $F$ has $\mathcal{M}$-negligible  
variation outside of $G^{o}$, 
and this ends the proof. 
\end{proof}

Theorem \ref{T2.1} together with Theorem 3(c) in \cite{GUO1} 
yields immediately the following statement. 

\begin{corollary}\label{C2.1}
Let $f: G \to X$ be a Dunford-McShane integrable function on $G$ with the primitive $F$ 
and 
let $h: G \to X$ be a Dunford-McShane integrable function on $G$ with the primitive $H$ 
Then, 
\begin{itemize}
\item[(i)]
$f + h$ is Dunford-McShane integrable function on $G$ with the primitive $F + H$,
\item[(ii)] 
$r.f$ is Dunford-McShane integrable function on $G$ with the primitive $r.F$, 
where $r \in \mathbb{R}$.  
\end{itemize}
\end{corollary}

The next theorem follows from Theorem \ref{T2.1} with $G = I_{0}$. 

\begin{theorem}\label{T2.11}
Let $f: I_{0} \to X$ be a function and let $F: \mathcal{I}_{I_{0}} \to X$ be an additive interval function.  
Then, $f$ is McShane integrable on $I_{0}$ with the primitive $F$ if and only if we have
\begin{itemize}
\item
for each $\varepsilon >0$ there exists a gauge $\delta_{\varepsilon}$ on $(I_{0})^{o}$ such that  
for each $\delta_{\varepsilon}$-fine $\mathcal{M}$-partition $\pi$ in 
$(I_{0})^{o}$, we have
$$
|| \sum_{(t,I) \in \pi } 
( f(t)|I| - F(I) ) || < \varepsilon,   
$$
\item
$F$ is a Dunford-function, 
\item
$F$ has $\mathcal{M}$-negligible variation on $Z = I_{0} \setminus (I_{0})^{o}$.
\end{itemize}
\end{theorem}

\begin{theorem}\label{T2.2}
Let $f: G \to X$ be a function and let $F: \mathcal{I}_{G} \to X$ be an additive interval function. 
Then, the following statements are equivalent:
\begin{itemize}
\item[(i)]
$f$ is  the Dunford-Henstock-Kurzweil integrable on $G$ with the primitive $F$,
\item[(ii)] 
$f_{0}$ is the Henstock-Kurzweil integrable on $I_{0}$ with the primitive $F_{0}$ 
such that 
\begin{equation}\label{eqHHK}
F_{0}(I) = \sum_{k:|I \cap C_{k}|>0} F(I \cap C_{k}), 
\text{ for all }I \in \mathcal{I}_{I_{0}}
\end{equation}
whenever $(C_{k}) \in \mathscr{D}_{G^{o}}$, 
and $F$ is  a Dunford-function,
\item[(iii)]  
$F$ is a Dunford-function and has $\mathcal{HK}$-negligible variation outside of $G^{o}$, and given any division $(C_{k}) \in \mathscr{D}_{G^{o}}$, 
each $f_{k}$ is Henstock-Kurzweil integrable on $C_{k}$ 
with the primitive $F_{k}$. 
\end{itemize}
\end{theorem}
\begin{proof}
By the same manner as in Theorem \ref{T2.1}, 
it can be proved $(i) \Rightarrow (iii)$ 
and $(iii) \Rightarrow (ii)$.

$(ii) \Rightarrow (i)$ 
Assume that $(ii)$ holds.  
Then, by Lemma 3.4.1 in \cite{SCH1} and by Lemma \ref{L2.1},  
given $\varepsilon >0$ there exists a gauge $\delta_{0}$ on $I_{0}$ such that 
for each $\delta_{0}$-fine $\mathcal{HK}$-partition $\pi$ in $I_{0}$, we have
\begin{equation}\label{eq_SaksHK}
||\sum_{(t,I) \in \pi } (~ f_{0}(t)|I| - F_{0}(I)~) || < \varepsilon,   
\end{equation}
and for each $\delta_{0}$-fine $(G \setminus G^{o})$-tagged 
$\mathcal{HK}$-partition $\pi$ in $I_{0}$, we have
\begin{equation}\label{eq_SaksHKK}
||\sum_{(t,I) \in \pi } f_{0}(t)|I|~ || < \varepsilon.   
\end{equation}
We can also choose $\delta_{0}$ 
so that $B_{m}(t,\delta_{0}(t)) \subset G^{o}$, 
for all $t \in G^{o}$. 
Hence, if we define $\delta = \delta_{0} \vert_{G^{o}}$, 
then for each 
$\delta$-fine $\mathcal{HK}$-partition $\pi$ in $G^{o}$, we have
\begin{equation*}
||\sum_{(t,I) \in \pi } (~ f(t)|I| - F(I)~) ||< \varepsilon.   
\end{equation*}
Thus, it remains to show that 
$F$ has $\mathcal{HK}$-negligible variation outside of $G^{o}$.  
To see this, define a gauge $\delta_{v}$ on $(G^{o})^{c}$ 
by $\delta_{v}(t) = \delta_{0}(t)$ if $t \in I_{0} \setminus G^{o}$, 
while for $t \notin I_{0}$, we choose $\delta_{v}(t)$ so that 
$B_{m}(t, \delta_{v}(t) ) \cap I_{0} = \emptyset$. 
Assume that $\pi_{v}$ is a  
$(G^{o})^{c}$-tagged $\delta_{v}$-fine $\mathcal{HK}$-partition 
in $\mathbb{R}^{m}$. 
Hence, 
$$
\pi_{0} = 
\{ (t,I \cap I_{0}): (t,I) \in \pi_{v}, t \in I_{0} \setminus G^{o}, |I \cap I_{0}| >0 \}
$$
is a $\delta_{0}$-fine $\mathcal{HK}$-partition in $I_{0}$. 
Note that $\pi_{0} = \pi_{a} \cup \pi_{b}$,  
where
$$
\pi_{a} = 
\{ (t,J) \in  \pi_{0}: t \in I_{0} \setminus G \}
\text{ and }
\pi_{b} = 
\{ (t,J) \in  \pi_{0}: t \in (G \setminus G^{o})\}.
$$
Since $\pi_{a}$ and $\pi_{b}$ are $\delta_{0}$-fine $\mathcal{HK}$-partitions in $I_{0}$, 
by \eqref{eq_SaksHK} and \eqref{eq_SaksHKK},  
it follows that
\begin{equation*}
\begin{split} 
|| \sum_{(t,J) \in \pi_{0}}   F_{0}(J) || 
&\leq
|| \sum_{(t,J) \in \pi_{a}}   F_{0}(J) ||  + 
|| \sum_{(t,J) \in \pi_{b}}   F_{0}(J) || \\
&\leq
|| \sum_{(t,J) \in \pi_{a}}   (f_{0}(t)|J| - F_{0}(J))  ||  + 
|| \sum_{(t,J) \in \pi_{b}}   (f_{0}(t)|J| - F_{0}(J))  || \\
&+|| \sum_{(t,J) \in \pi_{b}}   f_{0}(t)|J|~ || < 3\varepsilon.
\end{split}
\end{equation*} 
On the other hand, by \eqref{eqHHK}, 
we have also 
\begin{equation*}
\begin{split} 
|| \sum_{(t,J) \in \pi_{0}}   F_{0}(J) || 
&= || \sum_{(t,J) \in \pi_{0}} 
\left ( \sum_{k:|J \cap C_{k}|>0} F(J \cap C_{k}) \right )|| \\
&=|| \sum_{(t,I) \in \pi_{v}} \left ( \sum_{k:|I \cap C_{k}|>0} F(I \cap C_{k})  \right )||,
\end{split}
\end{equation*} 
whenever $(C_{k}) \in \mathscr{D}_{G^{o}}$. 
It follows that
\begin{equation*}
\begin{split} 
|| \sum_{(t,I) \in \pi_{v}} \left ( \sum_{k:|I \cap C_{k}|>0} F(I \cap C_{k}) \right )|| 
< 3\varepsilon. 
\end{split}
\end{equation*} 
This means that $F$ has $\mathcal{HK}$-negligible  
variation outside of $G^{o}$, 
and this ends the proof. 
\end{proof}

It easy to see that Theorem \ref{T2.2} together with Theorem 3.3.6 in \cite{SCH1} 
yields the following statement.

\begin{corollary}\label{C2.2}
Let $f: G \to X$ be a Dunford-Henstock-Kurzweil  integrable function on $G$ with the primitive $F$ 
and 
let $h: G \to X$ be a Dunford-Henstock-Kurzweil  integrable function on $G$ with the primitive $H$ 
Then, 
\begin{itemize}
\item[(i)]
$f + h$ is Dunford-Henstock-Kurzweil  integrable function on $G$ with the primitive $F + H$,
\item[(ii)] 
$r.f$ is Dunford-Henstock-Kurzweil integrable function on $G$ with the primitive $r.F$, 
where $r \in \mathbb{R}$.  
\end{itemize}
\end{corollary}

The next theorem follows from Theorem \ref{T2.2} with $G = I_{0}$.

\begin{theorem}\label{T2.21}
Let $f: I_{0} \to X$ be a function and let $F: \mathcal{I}_{I_{0}} \to X$ be a Dunford-function.  
Then, $f$ is Henstock-Kurzweil integrable on $I_{0}$ with the primitive $F$,  
if and only if we have
\begin{itemize}
\item
for each $\varepsilon >0$ there exists a gauge $\delta_{\varepsilon}$ on $(I_{0})^{o}$ such that  
for each $\delta_{\varepsilon}$-fine $\mathcal{HK}$-partition $\pi$ in 
$(I_{0})^{o}$, we have
$$
|| \sum_{(t,I) \in \pi } 
( f(t)|I| - F(I) ) || < \varepsilon,   
$$
\item
$F$ has $\mathcal{HK}$-negligible variation on $Z = I_{0} \setminus (I_{0})^{o}$.
\end{itemize}
\end{theorem}

Finally, we are going to prove a relationship between 
the Dunford-Henstock-Kurzweil and Dunford-McShane integrals 
in terms of the Dunford integral.

\begin{theorem}\label{T2.22}
Let $f: I_{0} \to X$ be a function and let $F: \mathcal{I}_{I_{0}} \to X$ be an additive interval function.  
Then, the following statements are equivalent:
\begin{itemize}
\item[(i)]
$f$ is Dunford integrable and Dunford-Henstock-Kurzweil integrable on $I_{0}$ with the primitive $F$,
\item[(ii)]
$f$ is Dunford-McShane integrable on $I_{0}$ with the primitive $F$. 
\end{itemize}
\end{theorem}
\begin{proof}
$(i) \Rightarrow (ii)$ Assume that $(i)$ holds. 

We first claim that 
\begin{equation}\label{eqDunf1}
(D)\int_{I} f d\lambda = x^{**}_{I} \in e(X), \text{ for all } I \in \mathcal{I}_{I_{0}}.
\end{equation}
To see this,  
let $I$ be an arbitrary closed non-degenerate interval in $\mathcal{I}_{I_{0}}$. 
Then, 
$$
x^{**}_{I}(x^{*}) = (M)\int_{I} x^{*}f d \lambda, \text{ for all } x^{*} \in X^{*},
$$
and since McShane integrable functions are Henstock-Kurweil integrable, 
it follows that
\begin{equation}\label{eqDunf2}
x^{**}_{I}(x^{*}) = (HK)\int_{I} x^{*}f d \lambda, \text{ for all } x^{*} \in X^{*}. 
\end{equation}
On the other hand, by Theorem \ref{T2.2}, 
$f$ is Henstock-Kurzweil integrable on $I_{0}$ with the primitive $F$ 
and, therefore, by Theorem 6.1.1 in \cite{SCH1}, it follows that
\begin{equation*}
(HK)\int_{I} x^{*}f d \lambda = x^{*} \left ( (HK)\int_{I} f d \lambda \right ), 
\text{ for all } x^{*} \in X^{*}.
\end{equation*}
The last result together with \eqref{eqDunf2} yields that \eqref{eqDunf1} holds.

We now claim that for any $(I_{k}) \in \mathscr{P}_{I_{0}}$ the series 
\begin{equation}\label{eqDunf3}
\sum_{k=1}^{+\infty} (D)\int_{I_{k}} f d\lambda 
\end{equation}
is norm convergent in $e(X)$. 
Indeed, since $F$ is a Dunford-function and the equality
$$
\sum_{k=1}^{+\infty} (D)\int_{I_{k}} f d\lambda = \sum_{k=1}^{+\infty} e(F(I_{k}))
$$
holds, it follows that \eqref{eqDunf3} is norm convergent in $e(X)$.

By virtue of Lemma 6 in \cite{GUO2}, $f$ is Pettis integrable. 
Thus, we have 
$f$ is Pettis integrable and  Henstock-Kurzweil integrable on $I_{0}$ with the primitive $F$. 
Therefore, by Theorem's Fremlin, Theorem 6.2.6 in \cite{SCH1}, 
we obtain that 
$f$ is McShane integrable on $I_{0}$ with the primitive $F$. 
Further, by Theorem \ref{T2.1}, 
it follows that 
$f$ is Dunford-McShane integrable on $I_{0}$ with the primitive $F$.

$(ii) \Rightarrow (i)$ Assume that $f$ is Dunford-McShane integrable on $I_{0}$ with the primitive $F$. 
Then, by Definition \ref{DH}, 
it follows that 
$f$ is Dunford-Henstock-Kurzweil integrable on $I_{0}$ with the primitive $F$.

By Theorem \ref{T2.1},  it follows that 
$f$ is McShane integrable on $I_{0}$ with the primitive $F$. 
Hence,  by Theorem 6.2.6 in \cite{SCH1}, 
$f$ is Pettis integrable and, therefore, $f$ is Dunford integrable,  
and this ends the proof.
\end{proof}

\bibliographystyle{plain}

\begin{thebibliography}{5}





\bibitem{BON} Bongiorno, B.,
\textit{The Henstock-Kurzweil integral, 
Handbook of measure theory},
Vol. I, II, 587-615, North-Holland, Amsterdam, (2002).






\bibitem{CAO} Cao, S.~S., 
\textit{The Henstock integral for Banach-valued functions}, 
SEA Bull. Math., \textbf{16} (1992), 35-40.







\bibitem{DIES} J.~Diestel and J.~J.~Uhl, 
\textit{Vector Measures}, 
Math. Surveys, vol. 15, Amer. Math. Soc., Providence, RI, 1977.







\bibitem{DIP} Di Piazza, L. and Musial, K.,
\textit{A characterization of variationally McShane integrable Banach-space valued functions},
Illinois J.Math.,\textbf{45} (2001), 279-289. 







\bibitem{FOLL} Folland, G.~B., 
\textit{Real Analysis}, 
Modern techniques and their applications.
Second edition. Pure and Applied Mathematics (New York).  
A Wiley-Interscience Publication. John Wiley $\&$ Sons, Inc., New York, (1999).






\bibitem{FRE1} Fremlin, D.~H., 
\textit{The Henstock and McShane integrals of vector-valued functions}, 
Illinois J.Math. \textbf{38} (1994), 471-479.






\bibitem{FRE2} Fremlin, D.~H., Mendoza, J., 
\textit{The integration of vector-valued functions}, 
Illinois J. Math. \textbf{38} (1994) 127-147.





\bibitem{FRE3} Fremlin, D.~H., 
\textit{The generalized McShane integral}, 
Illinois J. Math. 39 (1995), 39-67.






\bibitem{GORD1} Gordon, R.~A.,  
\textit{The McShane integral of Banach-valued functions}, 
Illinois J. Math. \textbf{34} (1990), 557-567.






\bibitem{GORD2} Gordon, R.~A.,  
\textit{The Denjoy extension of the Bochner, Pettis, and Dunford integrals}, 
Studia Math. T.XCII (1989), 73-91.






\bibitem{GORD3} Gordon, R.~A.,  
\textit{The Integrals of Lebesgue, Denjoy, Perron, and Henstock}, 
Amer. Math. Soc., (1994).






\bibitem{GUO1} Guoju, Y., and Schwabik, \v S., 
\textit{The McShane integral and the Pettis integral of Banach space-valued functions defined
on $\mathbb{R}^{m}$}, 
Illinois J. Math. \textbf{46} (2002), 1125-1144.







\bibitem{GUO2} Guoju, Y.,
\textit{On Henstock-Kurzweil and McShane integrals of Banach space-valued functions}, 
J. Math. Anal. Appl. \textbf{330} (2007), 753-765. 





\bibitem{KAL1} Kaliaj.~S.~B., 
\textit{The New Extensions of the Henstock-Kurzweil and the McShane Integrals of Vector-Valued Functions}, 
Mediterr. J. Math. (2018) 15: 22. https://doi.org/10.1007/s00009-018-1067-2.




\bibitem{KURZ} Kurzweil, J., Schwabik, \v S., 
\textit{On the McShane integrability of Banach space-valued functions}, 
Real Anal. Exchange \textbf{2} (2003-2004), 763-780.







\bibitem{LEE1} Lee P.~Y, 
\textit{Lanzhou Lectures on Henstock Integration}, 
Series in Real Analysis 2, World Scientific Publishing Co., Inc., (1989).







\bibitem{LEE2} Lee P.~Y, and V\'yborn\'y, R.,  
\textit{The Integral: An Easy Approach after
Kurzweil and Henstock}, 
Australian Mathematical Society Lecture Series 14, 
Cambridge University Press, Cambridge, (2000).






\bibitem{TUO}  Lee.~T.~Y.,  
\textit{Henstock-Kurzweil Integration on Euclidean Spaces}, 
Series in Real Analysis 12, 
World Scientific Publishing Co. Pte. Ltd., (2011).




\bibitem{MCSH} McShane, E.~J.,
\textit{Unifed integration}, 
Academic Press, San Diego, 1983.














\bibitem{SCH1} Schwabik, \v S.  and Guoju, Y., 
\textit{Topics in Banach Space Integration}, 
Series in Real Analysis, vol. 10, World Scientific, Hackensack, NJ, (2005). 











\end{thebibliography}

\end{document}